\documentclass{lmcs}
\usepackage[utf8]{inputenc}
\pdfoutput=1

\usepackage{lastpage}
\lmcsdoi{18}{3}{20}
\lmcsheading{}{\pageref{LastPage}}{}{}%
{Jun.~04,~2021}{Aug.~09,~2022}{}

\usepackage{amssymb}
\usepackage{stmaryrd} 
\usepackage{tikz}
\usepackage[all]{xy}

\newcommand{\defemph}[1]{\emph{\textbf{#1}}} 

\renewcommand{\AA}{\mathbb{A}} 
\newcommand{\effAA}{\mathbb{A}'} 
\newcommand{\combK}{\mathtt{k}} 
\newcommand{\combS}{\mathtt{s}} 
\newcommand{\app}{{\cdot}} 
\newcommand{\numeral}[1]{\overline{#1}} 

\newcommand{\NN}{\mathbb{N}} 
\newcommand{\RR}{\mathbb{R}} 

\newcommand{\lthen}{\Rightarrow} 
\newcommand{\liff}{\Leftrightarrow} 

\newcommand{\of}{{:}} 

\newcommand{\predicate}[2]{#1 \subseteq #2} 
\newcommand{\ileq}{\leq_{\mathsf{I}}} 
\newcommand{\iequiv}{\equiv_{\mathsf{I}}} 
\newcommand{\ibot}{{\bot\!\!\!\bot}} 
\newcommand{\itop}{{\top\!\!\!\top}} 
\newcommand{\iinf}{\sqcap} 
\newcommand{\isup}{\sqcup} 
\newcommand{\ithen}{\sqsupset} 
\newcommand{\ibiginf}[1]{{\textstyle\bigsqcap_{#1}}} 
\newcommand{\ibiginfbase}[1]{{\textstyle\bigsqcap'_{#1}}} 
\newcommand{\ibigsup}[1]{{\textstyle\bigsqcup_{#1}}} 

\newcommand{\true}[1]{\top_{\!#1}} 
\newcommand{\param}[2]{#1^{(#2)}} 

\newcommand{\LEM}{\mathsf{LEM}} 
\newcommand{\DNE}{\mathsf{DNE}} 
\newcommand{\WLEM}{\mathsf{WLEM}} 
\newcommand{\LPO}{\mathsf{LPO}} 
\newcommand{\CT}{\mathsf{CT}}   

\newcommand{\truthdeg}[1]{{#1}^{\star}}
\newcommand{\extent}[1]{\|#1\|}
\newcommand{\support}[1]{\|#1\|}

\newcommand{\up}[1]{{\uparrow}#1} 
\newcommand{\upp}[1]{\mathcal{U}(#1)} 
\newcommand{\sleq}{\leq_{\mathsf{S}}} 

\newcommand{\wleq}{\leq_{\mathsf{w}}} 
\newcommand{\wequiv}{\equiv_{\mathsf{w}}} 
\newcommand{\Wleq}{\leq_{\mathsf{W}}} 
\newcommand{\Wequiv}{\equiv_{\mathsf{W}}} 

\newcommand{\one}{\mathsf{1}} 
\newcommand{\two}{\mathsf{2}} 
\newcommand{\Cantor}{ {\two^\NN } } 

\newcommand{\Baire}{\NN^\NN} 
\newcommand{\effBaire}{(\Baire)_\mathrm{eff}} 

\newcommand{\im}[1]{\mathrm{im}(#1)} 

\newcommand{\Set}{\mathsf{Set}}

\newcommand{\RT}[1]{\mathsf{RT}(#1)}

\newcommand{\Asm}[1]{\mathsf{Asm}(#1)}

\newcommand{\rz}{\Vdash} 

\newcommand{\Pred}[1]{\mathrm{Pred}(#1)} 
\newcommand{\rpredicate}[2]{#1 \in \Pred{#2}} 

\newcommand{\id}{{\rm{id}}}             
\newcommand{\epito}{\twoheadrightarrow} 
\newcommand{\parto}{\rightharpoonup}    

\newcommand{\set}[1]{\{#1\}}
\newcommand{\such}{\mid}

\newcommand{\pow}[1]{\mathcal{P}(#1)} 
\newcommand{\powinh}[1]{\mathcal{P}_{+}(#1)} 

\newcommand{\pAA}{\pow{\AA}}
\newcommand{\ppAA}{\pow{\pow{\AA}}}

\newcommand{\extw}[1]{\widehat{#1}} 

\newcommand{\invim}[1]{#1^{*}}
\newcommand{\allim}[1]{\forall_{#1}}
\newcommand{\someim}[1]{\exists_{#1}}

\newcommand{\pair}[1]{\langle #1 \rangle}
\newcommand{\fst}{\mathtt{fst}}
\newcommand{\snd}{\mathtt{snd}}

\newcommand{\defined}[1]{#1{\downarrow}}

\newcommand{\all}[1]{\forall #1 \,.\,}
\newcommand{\some}[1]{\exists #1 \,.\,}

\title[Instance reducibility and Weihrauch degrees]{Instance reducibility and Weihrauch degrees}
\author[A.~Bauer]{Andrej Bauer\lmcsorcid{0000-0001-5378-0547}}
\thanks{This material is based upon work supported by the Air Force Office of Scientific Research under award number FA9550-21-1-0024.}
\address{Faculty of Mathematics and Physics, University of Ljubljana, Jadranska 19, 1000 Ljubljana, Slovenia}
\address{Institute of Mathematics, Physics and Mechanics, Jadranska 19, 1000 Ljubljana, Slovenia}
\email{Andrej.Bauer@andrej.com}

\begin{document}

\begin{abstract}
  We identify a notion of reducibility between predicates, called \emph{instance reducibility}, which commonly appears in reverse constructive mathematics. The notion can be generally used to compare and classify various principles studied in reverse constructive mathematics (formal Church's thesis, Brouwer's Continuity principle and Fan theorem, Excluded middle, Limited principle, Function choice, Markov's principle, etc.).
  We show that the instance degrees form a frame, i.e., a complete lattice in which finite infima distribute over set-indexed suprema. They turn out to be equivalent to the frame of upper sets of truth values, ordered by the reverse Smyth partial order. We study the overall structure of the lattice: the subobject classifier embeds into the lattice in two different ways, one monotone and the other antimonotone, and the $\lnot\lnot$-dense degrees coincide with those that are reducible to the degree of Excluded middle.

  We give an explicit formulation of instance degrees in a relative realizability topos, and call these
  \emph{extended Weihrauch degrees}, because in Kleene-Vesley realizability the $\lnot\lnot$-dense modest instance degrees correspond precisely to Weihrauch degrees.
  The extended degrees improve the structure of Weihrauch degrees by equipping them with computable infima and suprema, an implication, the ability to control access to parameters and computation of results, and by generally widening the scope of Weihrauch reducibility.
\end{abstract}

\maketitle

\section{Introduction}%
\label{sec:introduction}

A common way of proving an implication between two universally quantified statements,
\begin{equation*}
  (\all{y \in B} \psi(y)) \implies (\all{x \in A} \phi(x)),
\end{equation*}
is to show that for every $x \in A$ there exists $y \in B$ such that $\psi(y)$ implies~$\phi(x)$.
The technique is prevalent in constructive reverse mathematics where we compare undecided universally quantified statements. For example, the total ordering of reals implies the Limited principle of omniscience (LPO),
\begin{equation}%
  \label{eq:lpo}
  (\all{x, y \in \RR} x \leq y \lor y < x)
  \implies
  \all{\alpha \in \Cantor} (\all{n} \alpha_n = 0) \lor (\some{n} \alpha_n = 1),
\end{equation}
by such a reduction: given $\alpha \in \Cantor$, construct the real
$
  x = \lim_{n \to \infty} 2^{-{\min}\set{k \in \NN \such \alpha_k = 1 \lor k = n}}
$
and note that $x \leq 0 \lor 0 < x$ implies $(\all{n} \alpha_n = 0) \lor (\some{n} \alpha_n = 1)$.
One can find many other examples of the method, which therefore deserves to be named.
Let us write $\predicate{\phi}{A}$ to indicate that~$\phi$ is a predicate on~$A$.

\begin{defi}%
  \label{def:instance-reducibility}
  A predicate~$\predicate{\phi}{A}$ is \defemph{instance reducible} to a
  predicate~$\predicate{\psi}{B}$, written $(\phi, A) \ileq (\psi, B)$ or just
  $\phi \ileq \psi$, when
  \begin{equation}%
    \label{eq:instance-reducibility}%
    \all{x \in A} \some{y \in B} \psi(y) \lthen \phi(x).
  \end{equation}
  We say that $\phi$ and~$\psi$  are \defemph{instance
    equivalent}, written $(\phi, A) \iequiv (\psi, B)$ or just $\phi \iequiv \psi$,
  when $\phi \ileq \psi$ and $\psi \ileq \phi$. The equivalence class of a
  predicate~$\phi$ with respect to~$\iequiv$ is its \defemph{instance degree}.
\end{defi}

It is clear that $\ileq$ is a preorder (reflexive and transitive) in
which harder problems are higher up. There is no shortage of instance degrees, just consider the principles that are commonly studied form a constructive point of view: Excluded Middle and its various special cases, such as the Limited Principle of Omniscience (LPO), Brouwer's Fan and Continuity Principles, Markov Principle, Church's Thesis, etc. The relationships between these are well-known~\cite{ishihara06:_rever_mathem_bishop_const_mathem}, and shall not be rehashed here.

\subsubsection*{Overview}%
\label{sec:overview}

In the first part of the paper (\autoref{sec:inst-reduc}), which presupposes basic familiarity with intuitionistic logic, we study instance reducibilities, show that they form a frame, and enjoy a rich structure.
The second part (\autoref{sec:inst-degr-rz}) is written classically and relies on the first part only superficially. It starts with a brief overview of relative realizability models, and continues with a calculation of an explicit description of instance reducibility in such models. We relate instance reducibility to Weihrauch reducibility and show that the former is a proper extension of the latter. We study several examples (\autoref{sec:examples}) that demonstrate how the extended Weihrauch reducibility increases the scope of the subject.


\section{Instance reducibility}%
\label{sec:inst-reduc}

In this section we work in intuitionistic logic without countable choice.
To be formally precise, the text may be interpreted in the internal language of an elementary topos~\cite{jim86:_introd_higher_order_categ_logic,lane92:_sheav_geomet_logic}, although we shall not rely on the topos-theoretic machinery, nor do we expect the reader to be familiar with it.

\subsection{Instance degrees and the Smyth preorder}%
\label{sec:inst-degr-smyth}

The instance degrees form a large preorder whose carrier is the proper class of all predicates on all sets. Let us show that the preorder is essentially small, i.e., it is equivalent to a small one.

Recall that a preorder $(L, {\leq_L})$ is a reflexive transitive relation. Its symmetrization $\equiv_L$ is the equivalence relation defined by $x \equiv_L y \iff x \leq_L y \land y \leq_L x$. The quotient $L/{\equiv_L}$ is partially ordered by the relation induced by $\leq_L$.
Preorders $(L, {\leq_L})$ and $(K, {\leq_K})$ are equivalent when they are equivalent as categories, i.e., there are monotone maps $f : L \to K$ and $g : K \to L$ such that $f (g (y)) \equiv_K y$ for all $y \in K$ and $g (f (x)) \equiv_L x$ for all $x \in L$. If $L$ and $K$ are equivalent then $L/{\equiv_L}$ and $K/{\equiv_K}$ are isomorphic.

In a partially ordered set $(L, {\le})$, the \defemph{upper closure} of $S \subseteq L$ is the set
\begin{equation*}
  \up{S} = \set{y \in L \such \some{x \in S} x \leq y}.
\end{equation*}
The (reverse) \defemph{Smyth preorder}~\cite{smyth78:_power_domain} on the power set~$\pow{L}$ is defined by
\begin{equation*}
  S \sleq T \iff \up{S} \subseteq \up{T},
\end{equation*}
or equivalently $\all{x \in S} \some{y \in T} y \leq x$.
(Beware, the Smyth preorder is normally defined in the opposite way, but we prefer to turn it upside down to match harder instances being higher up.)

We write $\Omega$ for the set of all truth values. It may be identified with the powerset $\pow{\set{\star}}$ of a singleton, with $\emptyset$ representing falsehood and $\set{\star}$ representing truth. Unless specified otherwise, a power set $\pow{L}$ is always partially ordered by~$\subseteq$.

\begin{prop}%
  \label{thm:smyth-equiv}
  The preorder of instance reducibilities is equivalent to the Smyth preorder on $\pow{\Omega}$.
\end{prop}

\begin{proof}
  A predicate $\predicate{\phi}{A}$ corresponds to its image
  \begin{equation*}
    \im{\phi} = \set{p \in \Omega \mid \some{x \in A} p = \phi(x)} \in \pow{\Omega}.
  \end{equation*}
  Conversely, $S \subseteq \Omega$ corresponds to the predicate $\predicate{\phi_S}{S}$ classified by the inclusion, $\phi_S(p) = p$.
  One checks easily that these correspondences are monotone with respect to instance reducibility and the Smyth preorder, and that they constitute an equivalence.
\end{proof}

The collection $\upp{L} = \set{ S \subseteq L \mid \up{S} = S }$ of the \defemph{upper sets} ordered by~$\subseteq$ is a frame in which infima and suprema are computed as intersections and unions, respectively. Upper closure $\up{} : \pow{L} \to \upp{L}$ is the poset reflection of the Smyth preorder.
The partial order of instance degrees is therefore isomorphic to the frame $\upp{\Omega}$, and in a sense that is all that needs to be said.
Nevertheless, let us spell out the ordered-theoretic structure of instance degrees directly in terms of predicates, and study it a bit more closely.

\subsection{Transfer of predicates along maps}%
\label{sec:inst-reduc-transf}

It often happens that a reduction $(\phi,A) \ileq (\psi,B)$ is accomplished by means
of a map $f : A \to B$ such that $\psi(f(x)) \lthen \phi(x)$ for all $x \in A$, in which
case we say that~$f$ \defemph{witnesses} the reduction $\phi \ileq \psi$. Not all reductions
need be witnessed by maps, for in general there may be no choice map
for~\eqref{eq:instance-reducibility}.

A predicate $\predicate{\psi}{B}$ may be transferred along a map $f : A \to B$ to a
predicate $\predicate{\invim{f}{\psi}}{A}$, defined by
\begin{equation*}
  (\invim{f}{\psi})(x) \iff \psi(f(x)).
\end{equation*}
In the other direction, a predicate $\predicate{\phi}{A}$ may be transferred to one on~$B$
in two ways, using either the existential or the universal quantifier, to give predicates
$\predicate{\allim{f}{\phi}}{B}$ and $\predicate{\someim{f}{\phi}}{B}$, defined by
\begin{align*}
  (\allim{f}{\phi})(y) &\iff
   \all{x \in A} f(x) = y \lthen \phi(x)
  \\
  (\someim{f}{\phi})(y) &\iff
  \some{x \in A} f(x) = y \land \phi(x).
\end{align*}

\begin{lem}%
  \label{lem:transfer-ileq}
  Let $f : A \to B$ be a map, $\predicate{\phi}{A}$ and $\predicate{\psi}{B}$. Then
  $\invim{f}{\psi} \ileq \psi$ and $\phi \ileq \allim{f}{\phi}$. If $f$ is surjective then also $\psi \ileq \invim{f}{\psi}$ and $\someim{f}{\phi} \ileq \phi$.
\end{lem}

\begin{proof}
  The first two reductions are witnessed by~$f$. If $f$ is surjective then we obtain
  the remaining two reductions by reducing $y \in B$ to $x \in A$ such that $f(x) = y$.
\end{proof}

When the map $f$ in the previous lemma is a projection $\pi_1 : A \times B \to B$ and $A$
is inhabited, we obtain for any $\predicate{\phi}{A \times B}$
\begin{equation*}
  (\someim{\pi_1}{\phi}) \ileq \phi \ileq (\allim{\pi_1}{\phi}),
\end{equation*}
which is more memorable with the abuse notation
\begin{equation*}
  (\some{x \in A} \phi(x,y))
  \; \ileq \;
  \phi(x,y)
  \; \ileq \;
  (\all{x \in A} \phi(x,y)).
\end{equation*}

\subsection{The order-theoretic structure}%
\label{sec:order-theor-struct}

By~\autoref{thm:smyth-equiv} the instance degrees form a frame, whose structure we describe explicitly in this section.

\begin{prop}%
  \label{prop:ileq-lattice}
  The instance degrees form a bounded distributive lattice.
\end{prop}

\begin{proof}
  The empty predicate $\predicate{\emptyset }{\emptyset}$ on the empty set is instance
  reducible to every other predicate and so it represents the smallest instance degree~$\ibot$.

  The largest instance degree $\itop$ is represented by the empty predicate
  $\predicate{\emptyset}{\one}$ on the singleton $\one = \set{\star}$, or more generally
  by any predicate $\phi \subseteq A$ with a \emph{counter-example}, which is $a \in A$ such that $\lnot \phi(a)$.

  The supremum of $\predicate{\phi}{A}$ and $\predicate{\psi}{B}$ is the predicate
  $\phi \isup \psi$, defined as follows. Let $A + B$ be the disjoint union with canonical
  inclusions $\iota_1 : A \to A + B$ and $\iota_2 : B \to A + B$. Define $\phi \isup \psi$
  on $A + B$ by
  \begin{align*}
    (\phi \isup \psi)(\iota_1(x)) \iff \phi(x)
    \qquad\text{and}\qquad
    (\phi \isup \psi)(\iota_2(y)) \iff \psi(y).
  \end{align*}
  The reductions of $\phi$ and $\psi$ to $\phi \isup \psi$ are witnessed by the canonical
  inclusions.
  %

  The infimum of $\predicate{\phi}{A}$ and $\predicate{\psi}{B}$ is the predicate
  $\predicate{\phi \iinf \psi}{A \times B}$, defined by
  \begin{equation*}
    (\phi \iinf \psi)(x, y) \iff \phi(x) \lor \psi(y).
  \end{equation*}
  The canonical projections witness reductions of $\phi \iinf \psi$ to $\phi$ and $\psi$.

  The distributivity laws are easily checked. We only indicate how to prove
  \begin{equation*}
    (\phi \isup \psi) \iinf \theta
    \ileq
    (\phi \iinf \theta) \isup (\psi \iinf \theta),
  \end{equation*}
  for all $\predicate{\phi}{A}$, $\predicate{\psi}{B}$, and $\predicate{\theta}{C}$. The
  reduction is witnessed by the canonical isomorphism
  $(A + B) \times C \to (A \times C) + (B \times C)$ so long as, for any $x \in A$,
  $y \in B$, and $z \in C$,
  \begin{equation*}
    (\phi(x) \lor \theta(z)) \land (\psi(x) \lor \theta(z))
    \qquad\text{implies}\qquad
    (\phi(x) \land \psi(y)) \lor \theta(z),
  \end{equation*}
  which is just distributivity of disjunction over conjunction.
\end{proof}

\begin{prop}
  The instance degrees form a Heyting algebra.
\end{prop}

\begin{proof}
  We only have to define implication.
  Given $\predicate{\phi}{A}$ and $\predicate{\psi}{B}$, define the set
  \begin{equation*}
    (A \ithen B) = \set{p \in \Omega \such \all{x \in A} \some{y \in B} \psi(y) \lthen \phi(x) \lor p},
  \end{equation*}
  and the predicate $\phi \ithen \psi$ on $A \ithen B$ by $(\phi \ithen \psi)(p) = p$.
  We need to show that, for all $\predicate{\theta}{C}$,
  \begin{equation*}
    \theta \iinf \phi \ileq \psi
    \quad\text{if and only if}\quad
    \theta \ileq \phi \ithen \psi.
  \end{equation*}
  To see that $\theta \iinf \phi \ileq \psi$ implies $\theta \ileq \phi \ithen \psi$, it suffices to check that for any $z \in C$ we have $\theta(z) \in A \ithen B$, which follows directly from the definition of~$\theta \iinf \phi$ and $\theta \iinf \phi \ileq \psi$.

  Conversely, suppose $\theta \ileq \phi \ithen \psi$.
  Consider any $z \in C$ and $x \in A$.
  There is $p \in A \ithen B$ such that $p \lthen \theta(z)$.
  By definition of $A \ithen B$ there is $y \in B$ such that $\psi(y) \lthen \phi(x)  \lor p$, and hence $\psi(y) \lthen \theta(z) \lor \phi(x)$, as required.
\end{proof}

\noindent
The preceding construction of implication is quite obviously just a dressed up version of implication in $\upp{\Omega}$. It would be desirable to have a more direct description.

\begin{prop}%
  \label{prop:suprema}
  For every set $I$, the instance degrees have $I$-indexed suprema.
\end{prop}

\begin{proof}
  Let $\predicate{\phi_i}{A_i}$ be a family of predicates indexed by $i \in I$. We claim
  that the supremum of the family is the predicate $\ibigsup{i \in I} \phi_i$ on the
  coproduct $\coprod_{i \in I} A_i$, defined by
  \begin{equation*}
    (\ibigsup{i \in I} \phi_i)(u) \iff
    \some{i \in I} \some{x \in A_i} u = \iota_i(x) \land \phi_i(x),
  \end{equation*}
  where $\iota_i : A_i \to \coprod_{i \in I} A_i$ is the $i$-th canonical inclusion.

  For any $j \in I$, the reduction $\phi_j \ileq \ibigsup{i \in I} \phi_i$ is witnessed
  by~$\iota_j$, so $\ibigsup{i} \phi_i$ is an upper bound of all the~$\phi_i$'s.
  To prove it is the least upper bound, consider any common upper bound
  $\predicate{\theta}{C}$ of the predicates $\phi_i$. Given any $u \in \coprod_i A_i$
  there are (unique) $j \in I$ and $x \in A_j$ such that $u = \iota_j(x)$. Because
  $\phi_j \ileq \theta$ there is $y \in C$ such that $\theta(y) \lthen \phi_j(x)$,
  therefore $\theta(y) \lthen (\ibigsup{i} \phi_i)(u)$ and hence
  $\ibigsup{i} \phi_i \ileq \theta$, as required.
\end{proof}

\begin{prop}%
  \label{prop:infima}
  For every set $I$, the instance degrees have $I$-indexed infima.
\end{prop}

\begin{proof}
  Let $\predicate{\phi_i}{A_i}$ be a family of predicates indexed by $i \in I$. For the
  underlying set of the infimum of the family we take the product
  $\prod_{i \in I} \powinh{A_i}$,
  where $\powinh{X}$ is the set of all inhabited subsets of~$X$. We define the
  predicate $\ibiginf{i \in I} \phi_i$ by
  \begin{equation}%
    \label{eq:ibiginf}
    (\ibiginf{i \in I} \phi_i)(f) \iff
    \some{i \in I} \some{x \in f(i)} \phi_i(x).
  \end{equation}

  For any $j \in I$, the reduction $\ibiginf{i \in I} \phi_i \ileq \phi_j$ is proved as
  follows. Given any $f \in \prod_{i \in I} \powinh{A_i}$, there exists $x \in f(j)$
  because $f(j)$ is inhabited. If $\phi_j(x)$ then $(\ibiginf{i \in I} \phi_i)(f)$
  by~\eqref{eq:ibiginf}.

  To prove that $\ibiginf{i \in I} \phi_i$ is the greatest lower bound, consider any
  common lower bound $\predicate{\psi}{B}$ of the predicates~$\phi_i$. The desired
  reduction $\psi \ileq \ibiginf{i \in I} \phi_i$ is witnessed by the map
  $r : B \to \prod_{i \in I} \powinh{A_i}$, defined by
  \begin{equation*}
    r(z)(i) = \set{ x \in A_i \such \phi_i(x) \lthen \psi(z)}.
  \end{equation*}
  For every $i \in I$, $r(z)(i)$ is inhabited because $\psi \ileq \phi_i$ and so there is
  $x \in A_i$ such that $\phi_i(x)$ implies $\psi(z)$. The map $r$ witnesses the reduction
  because, for any $z \in B$, if $(\ibiginf{i \in I} \phi_i)(r(z))$ then there exists
  $i \in I$ and $x \in r(z)(i)$ such that $\phi_i(x)$, but then $\psi(z)$ follows because
  $\phi_i(x) \lthen \psi(z)$ by the definition of~$r$.
\end{proof}




\subsection{Instance reducibilities and projective sets}%
\label{sec:cons-aczels-pres}

Recall that $I$ is a \defemph{projective} set when every $I$-indexed family ${(A_i)}_{i \in I}$ of inhabited sets has a choice function, i.e., the product $\prod_{i \in I} A_i$ is
inhabited. Equivalently, a set~$I$ is projective if every total relation on~$I$ has a
choice function:
\begin{equation*}
  (\all{i \in I} \some{x \in A} \phi(i, x))
  \lthen
  \some{f : A^I} \all{i \in I} \phi(i, f(i)).
\end{equation*}
In~\autoref{sec:inst-degr-rz} we shall interpret instance reducibility in
realizability toposes, which validate Aczel's \defemph{presentation
  axiom}~\cite{aczel78:_type_theor_inter_const_set_theor}.
\emph{``Every set is the image of a projective set.''}
It will be useful to know a couple of consequence of Aczel's axiom.

\begin{prop}%
  \label{prop:pa-equivalent-to-projective}
  If the presentation axiom holds then every predicate is instance equivalent to a
  predicate on a projective set.
\end{prop}

\begin{proof}
  If $\predicate{\phi}{A}$ and $e : B \epito A$ is a cover of $A$ by a projective set $B$,
  then $\phi \iequiv \invim{e}{\phi}$ by \autoref{lem:transfer-ileq}.
\end{proof}

When an infimum is indexed by a projective set, a simpler formula than the one given in
\autoref{prop:infima} can be used. Given a family of predicates $\predicate{\phi_i}{A_i}$
indexed by $i \in I$ with~$I$ projective, define the predicate
$\ibiginfbase{i \in I} \phi_i$ on the product $\prod_{i \in I} A_i$ by
\begin{equation*}
  ({\ibiginfbase{i \in I} \phi_i})(f) \iff
  \some{i \in I} \phi_i(f(i)).
\end{equation*}
We claim that $\ibiginfbase{i \in I} \phi_i \iequiv \ibiginf{i \in I} \phi_i$. The
reduction $\ibiginfbase{i \in I} \phi_i \ileq \ibiginf{i \in I} \phi_i$ is witnessed by
the map which takes $f \in \prod_{i \in I} A_i$ to the map $i \mapsto \set{f(i)}$. For the
opposite reduction, consider any $g \in \prod_{i \in I} \powinh{A_i}$. Because~$I$ is
projective there is a choice map $c \in \prod_{i \in I} A_i$ such that $c(i) \in g(i)$ for
all $i \in I$. If there is $j \in I$ such that $\phi_j(c(j))$ then
$(\ibiginf{i \in I} A_i)(g)$ holds because $j$ and $c(j)$ witness~\eqref{eq:ibiginf}
for~$g$.

\subsection{Reduction to many instances}%
\label{sec:param-reduct-many}

Occasionally a single instance of $\phi$ reduces to \emph{many} instances of $\psi$. Such variants are captured by the following operation.

\begin{defi}
  For any set $I$ and a predicate $\predicate{\phi}{A}$ define the
  \defemph{$I$-parameterization of $\phi$} to be the predicate $\predicate{\param{\phi}{I}}{A^I}$
  defined on the set $A^I$ of all functions from $I$ to $A$ by
  \begin{equation*}
    \param{\phi}{I}(f) \iff \all{i \in I} \phi(f(i)).
  \end{equation*}
\end{defi}

\begin{prop}
  The following hold for all $\predicate{\phi}{A}$ and $\predicate{\psi}{B}$:
  \begin{enumerate}
  \item If $I$ is inhabited then $\phi \ileq \param{\phi}{I}$.
  \item If $I$ is a retract of $J$ then $\param{\phi}{I} \leq \param{\phi}{J}$.
  \item If $I$ is projective then $\phi \ileq \psi$ implies $\param{\phi}{I} \ileq \param{\psi}{I}$.
  \end{enumerate}
\end{prop}

\begin{proof}
  The first statement is witnessed by the map taking $x \in A$ to the constant map $i \mapsto x$.
  The second statement is witnessed by the map taking $f \in A^I$ to $f \circ r \in A^J$,
  where $r : J \to I$ is a retraction, so that it has a right inverse $s : I \to J$.
  To verify the third statement, consider any $f \in A^I$. Since $I$ is projective and $\phi \ileq \psi$ is assumed, there exists $g : I \to B$ such that $\psi(g(i)) \lthen \phi(f(i))$ for all $i \in I$, hence $\param{\psi}{I}(g) \lthen \param{\phi}{I}(f)$.
\end{proof}

A reduction of $\predicate{\phi}{A}$ to $\predicate{\psi}{B}$ which uses a \emph{fixed}
number~$n$ of instances of $\psi$ for every instance of $\phi$ is just instance
reducibility of $\phi$ to $\param{\psi}{[n]}$ where $[n] = \set{1, \ldots, n}$ because
$\phi \ileq \param{\psi}{[n]}$ unfolds to
\begin{equation*}
  \all{x \in A}
  \some{y_1, \ldots, y_n \in B}
  \psi(y_1) \land \cdots \land \psi(y_n)
  \lthen
  \phi(x).
\end{equation*}
This is not to be confused with the reduction of $\phi(x)$ to a \emph{variable} finite
number~$n$ of instances $\psi(y_1), \ldots, \psi(y_n)$, which is expressed as
\begin{equation*}
  \phi \ileq \ibigsup{n \in \NN} \param{\psi}{[n]},
\end{equation*}
as it unfolds to
\begin{equation*}
  \all{x \in A}
  \some{n \in \NN}
  \some{y_1, \ldots, y_n \in B}
  \psi(y_1) \land \cdots \land \psi(y_n)
  \lthen
  \phi(x).
\end{equation*}
A third possibility is $\phi \ileq \param{\psi}{\NN}$ which reduces an instance $\phi(x)$ to
countably many instances $\psi(y_0), \psi(y_1), \psi(y_2), \ldots$

In the Smyth preorder on $\pow{\Omega}$ the $I$-parameterization is easily seen to amount to passing to $I$-indexed intersections: the $I$-parameterization of $S \in \pow{\Omega}$ is
\begin{equation*}
  \param{S}{I} = \set{ p \in \Omega \mid \some{f \in A^I} p \liff \all{i \in I} f(i)}.
\end{equation*}

\subsection{Embedding of truth values}%
\label{sec:truth-values}

Given any set $A$ let $\true{A}$ be the always true predicate on~$A$, and let
$\one = \set{\star}$ be the singleton set containing the element~$\star$. For each truth
value $p \in \Omega$ define its \defemph{extent} $\extent{p} = \set{x \in \one \such p}$.

\begin{prop}%
  \label{prop:omega-monotone}
  The principal ideal of instance degrees below $\true{\one}$ is equivalent to the
  poset~$\Omega$ of truth values, ordered by implication~$\lthen$.
\end{prop}

\begin{proof}
  First observe that, for a predicate $\predicate{\phi}{A}$, the reduction
  $\phi \ileq \true{\one}$ is equivalent to $\all{x \in A} \phi(x)$, therefore the
  predicates below $\true{\one}$ are precisely those of the form $\true{A}$.

  The desired equivalence of posets is established by the maps
  \begin{equation*}
    p \mapsto \true{\extent{p}}
    \qquad\text{and}\qquad
    \true{A} \mapsto (\some{x \in A} \top)
  \end{equation*}
  of truth values to instance degrees, and vice versa. Indeed, the maps are monotone with
  respect to $\lthen$ and $\ileq$, and it is easy to check that
  \begin{equation*}
    p \liff (\some{x \in \extent{p}} \top)
    \qquad\text{and}\qquad
    \true{\extent{\some{x \in A} \top}} \iequiv \true{A}.
    \qedhere
  \end{equation*}
\end{proof}

\medskip

There is another embedding of truth values into instance degrees, which is an anti-monotone
semilattice homomorphism.

\begin{prop}%
  \label{prop:omega-antimonotone}
  The map which takes a truth value $p \in \Omega$ to the predicate
  $\predicate{\truthdeg{p}}{\one}$, defined by $\truthdeg{p}(x) = p$, is an
  anti-monotone embedding of~$\Omega$ into instance degrees. It satisfies,
  for all $p, q \in \Omega$ and predicates $\predicate{\phi}{A}$:
  \begin{align*}
    \truthdeg{\top}                  &\;\iequiv\;  \true{\one}, \\
    \truthdeg{\bot}                  &\;\iequiv\;  \itop, \\
    \truthdeg{(p \lor q)}            &\;\iequiv\;  \truthdeg{p} \iinf \truthdeg{q}, \\
    \truthdeg{(\some{x \in A} \phi(x))} &\;\iequiv\;  \ibiginf{x \in A} \truthdeg{\phi(x)}.
  \end{align*}
\end{prop}

\begin{proof}
  The map is is an anti-monotone embedding because $\truthdeg{p} \ileq \truthdeg{q}$ is
  equivalent to $q \lthen p$. The four equivalences are easily checked, one just has to
  unfold the definitions.
\end{proof}

We can ask what it would take to also have
\begin{align*}
  \truthdeg{(p \land q)} &\;\iequiv\; \truthdeg{p} \isup \truthdeg{q}, \\
  \truthdeg{(\all{x \in A} \phi(x))} &\;\iequiv\; \ibigsup{x \in A} \truthdeg{\phi(x)}.
\end{align*}
The first equivalence is a special case of the second one, and the second one turns out to
be an instance of the so-called ``Drinker Paradox'', provided that $A$ is inhabited.

\begin{prop}
  For a predicate $\predicate{\phi}{A}$ on an inhabited set~$A$, the equivalence
  \begin{equation*}
    \truthdeg{(\all{x \in A} \phi(x))} \iequiv \ibigsup{x \in A} \truthdeg{\phi(x)}
  \end{equation*}
  holds if, and only if,
  \begin{equation}%
    \label{eq:drinker}
    \some{x \in A} (\phi(x) \lthen \all{y \in A} \phi(y)).
  \end{equation}
\end{prop}

\begin{proof}
  The reduction
  $\ibigsup{x \in A} \truthdeg{\phi(x)} \ileq \truthdeg{(\all{x \in A} \phi(x))}$ just
  holds. When we unfold the opposite reduction
  $\truthdeg{(\all{x \in A} \phi(x))} \ileq \ibigsup{x \in A} \truthdeg{\phi(x)}$
  we obtain
  \begin{equation*}
    \some{x \in A}
    \some{u \in \one}
    (\phi(x) \lthen \all{y \in A} \phi(y)),
  \end{equation*}
  which is equivalent to~\eqref{eq:drinker}.
\end{proof}

\subsection{\texorpdfstring%
{Excluded middle and the $\lnot\lnot$-dense degrees}%
{Excluded middle and the not-not-dense degrees}%
}%
\label{sec:not-not-dense}

We do not yet possess a method for showing that two given instance degrees actually
disagree. For all we know, the instance degrees might collapse to a very uninteresting
lattice.

\begin{prop}%
  \label{prop:when-trivial}
  The following are equivalent:
  \begin{enumerate}
  \item excluded middle,
  \item every instance degree is equivalent to either $\ibot$, $\true{\one}$, or $\itop$,
  \item instance reducibility is a total order,
  \item every instance degree is either below or above $\true{\one}$.
  \end{enumerate}
\end{prop}

\begin{proof}
  We prove the cycle of implications in the given order.

  Assume the law of excluded middle and consider any $\predicate{\phi}{A}$. If $A$ is empty then $\phi \iequiv \ibot$. Otherwise, $A$ is inhabited, in which case: if $\phi$ has a counter-example then $\phi \iequiv \itop$, and if it does not then $\phi \iequiv \true{\one}$.

  The second statement implies the third one because $\ibot \ileq \true{\one} \ileq \itop$.

  The fourth claim is an instance of the third one.

  Finally, suppose every instance degree is either below or above~$\true{\one}$.
  To decide $p \in \Omega$, compare $\true{\one}$ with the false predicate on $\extent{p}$.
  If $\true{\one}$ is above it then $p$, and if $\true{\one}$ is below it then $\lnot p$.
\end{proof}

A predicate $\predicate{\phi}{A}$ represents the largest instance degree precisely when it
has a counter-example, i.e., an element $a \in A$ such that $\lnot \phi(a)$. All the other
predicates are know as the \emph{$\lnot\lnot$-dense} ones.

\begin{defi}
  A predicate $\predicate{\phi}{A}$ is \defemph{$\lnot\lnot$-dense} when it has no
  counter-example, i.e., $\lnot \some{x \in A} \lnot \phi(x)$, or equivalently
  $\all{x \in A} \lnot\lnot \phi(x)$. A \defemph{$\lnot\lnot$-dense instance degree} is one
  represented by a $\lnot\lnot$-dense predicate.
\end{defi}

If a predicate does not represent the largest instance degree, then it is $\lnot\lnot$-dense. However, a predicate which is not $\lnot\lnot$-dense need not be the largest instance degree because intuitionistically $\lnot\lnot \some{x \in A} \lnot \phi(x)$ does not generally imply $\some{x \in A} \lnot \phi(x)$.


The $\lnot\lnot$-dense degrees are closely related to the degree of
\defemph{excluded middle}, which is the predicate $\predicate{\LEM}{\Omega}$ defined by $\LEM(p) \equiv (p \lor \lnot p)$. Its degree is the same as that of \defemph{double negation elimination} $\predicate{\DNE}{\Omega}$, defined by $\DNE(p) \equiv (\lnot\lnot p \lthen p)$. Indeed, the usual proofs of implications between the two principles are instance reductions.

\begin{prop}%
  \label{prop:dense-iff-below-lem}%
  An instance degree is $\lnot\lnot$-dense if, and only if, it is reducible to the degree of excluded middle.
\end{prop}

\begin{proof}
  It is simpler to establish the claim by using~$\DNE$ instead of~$\LEM$.
  If $\predicate{\phi}{A}$ is $\lnot\lnot$-dense then $\phi \ileq \DNE$ is witnessed by the map $A \to \Omega$ taking~$x$ to~$\phi(x)$.
  For the converse, suppose $\predicate{\phi}{A}$ is instance reducible to~$\DNE$. Then for any $x \in A$ there is $p \in \Omega$ such that $(\lnot\lnot p \lthen p) \lthen \phi(x)$, hence $\lnot \phi \lthen \lnot (\lnot\lnot p \lthen p)$, but $\lnot (\lnot\lnot p \lthen p)$ is false.
\end{proof}



%
\begin{figure}[htbp]
  \centering
  \begin{tikzpicture}[scale=3.2]
    \path[draw, fill=black!5]
        (0, -1) .. controls (0.5, -0.75) and (0.5, -0.25) ..
        (0, 0) .. controls (-0.5, -0.25) and (-0.5, -0.75) .. (0, -1) ;
    \node at (0, -0.5) {$\Omega$} ;
    \path[draw, fill=black!5]
        (0, 0) .. controls (0.5, 0.25) and (0.5, 0.75) .. (0, 1)
               .. controls (-0.5, 0.75) and (-0.5, 0.25) .. (0, 0) ;
    \node at (0,0) [above] {$\true{\one}$};
    \node at (0, 0.5) {$\mho$} ;
    \node at (0, -1) {$\bullet$} ; \node at (0, -1) [below] {$\ibot$} ;
    \node at (0, 0) {$\bullet$} ;
    \node at (0, 0.7) {$\bullet$} ;
    \node at (0, 1) {$\bullet$} ; \node at (0,1) [above] {$\itop$} ;
    \node at (0.65, 0.0) [left] {$\lnot\exists\lnot$};
    \path[draw]
        (0, -1) .. controls (0.5, -0.75) and (0.6, -0.25) ..
        (0, 0.7)  node[above] {$\LEM$}
                .. controls (-0.6, -0.25) and (-0.5, -0.75) .. (0, -1) ;
  \end{tikzpicture}
  \caption{The lattice of instance degrees}%
  \label{fig:instances}
\end{figure}
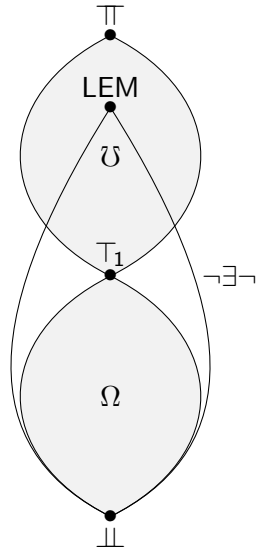

The rudimentary structure of instance degrees is summarized in \autoref{fig:instances}.
The gray areas marked with $\Omega$ and $\mho$ are the monotone and antimonotone embedding of~$\Omega$ from \autoref{prop:omega-monotone} and \autoref{prop:omega-antimonotone}, and the area marked with $\lnot\exists\lnot$ represents the $\lnot\lnot$-dense predicates.
The precise figure depends on what model of intuitionistic mathematics we look at.
In classical set theory, which validates excluded middle, the picture collapses to just the three dots representing $\ibot$, $\true{\one}$ and $\itop$,
while in a properly intuitionistic topos all areas are plentiful.
For example, in the Kleene-Vesley realizability topos, $\Omega$ and $\mho$ are the Medvedev lattice (\autoref{sec:embedd-truth-valu}), the area $\lnot\exists\lnot$ contains the Weihrauch lattice (\autoref{sec:weihr-degr-modest}), and ``Church's thesis'' degree $\CT$ from \autoref{sec:non-trivial-non} resides in $\mho$ but outside $\lnot\exists\lnot$.


\section{Instance reducibility in realizability models}%
\label{sec:inst-degr-rz}

We next work out the interpretation of instance degrees in realizability models. For this purpose we give a brief overview of the relevant concepts and refer to~\cite{oosten08:_realiz} for background material. Henceforth we work in classical mathematics.

\subsection{An overview of realizability models}%
\label{sec:an-overv-real}

A \defemph{partial combinatory algebra (pca)} is a set~$\AA$ equipped with a partial binary operation~$\app$ such that there exist the $\combK, \combS \in \AA$ satisfying, for all $a, b, c \in \AA$,
\begin{equation*}
  \defined{\combK \app a},
  \qquad
  \combK \app a \app b = a,
  \qquad
  \defined{\combS \app a \app b},
  \quad\text{and}\quad
  \combS \app a \app b \app c = (a \app c) \app (b \app c).
\end{equation*}
Application associates to the left, i.e., $a \app b \app c = (a \app b) \app c$.
Above we wrote $\defined{e}$ to indicate that the term~$e$ is defined, and therefore so are all the subterms.
If we use a particular term~$e$ in a statement, we tacitly assume that it is defined.
Partial combinatory algebras are so called because of their combinatorial completeness: given any term~$e$ formed using partial application, elements of~$\AA$ and a variable $x$, we can construct $a \in \AA$ such that, for all $b \in \AA$, if $\defined{e[b/x]}$ then $\defined{a \app b}$ and $a \app b = e[b/x]$, and similarly for several variables. The element~$a$ so constructed is denoted $[x] e$, for example $[x] x$ is the identity combinator which satisfies $([x] x) \app b = b$.
Every pca has a \defemph{pairing} and \defemph{projections} $\mathtt{pair}, \fst, \snd \in \AA$, which satisfy
\begin{equation*}
  \fst \app (\mathtt{pair} \app a \app b) = a
  \quad\text{and}\quad
  \snd \app (\mathtt{pair} \app a \app b) = b.
\end{equation*}
We also write $\pair{a,b}$ for $\mathtt{pair}\,a\,b$. A non-trivial pca has an embedding $\NN \to \AA$ which assigns to every number~$n$ its \defemph{numeral} $\numeral{n}$ in such a way that all partial computable maps are representable in~$\AA$.
Some well-known pcas are:
\begin{enumerate}
\item Kleene's first pca: the natural numbers $\NN$ with application $\{m\} n$ which applies the $m$-th partial computable function to~$n$;
\item Kleene's second pca: the Baire space $\Baire$ with application $\alpha \,{\mid}\, \beta$, as described in~\cite[\S1.4.3]{oosten08:_realiz};
\item the untyped $\lambda$-calculus is a pca whose elements are the closed terms, with $\combK = \lambda x \, y \,.\, x$ and $\combS = \lambda x \, y \, z \,.\, (x\,z)\,(y\,z)$.
\end{enumerate}
An \defemph{elementary sub-pca} is a subset $\effAA \subseteq \AA$ which is closed under application and contains some $\combK$ and $\combS$ suitable for~$\AA$. For example, the computable sequences $\effBaire \subseteq \Baire$ form an elementary sub-pca of Kleene's second pca~\cite[\S2.6.9]{oosten08:_realiz}.

From any pca $\AA$ with an elementary sub-pca $\effAA \subseteq \AA$, the \defemph{relative realizability topos} $\RT{\AA, \effAA}$ can be built~\cite[\S2.6.9]{oosten08:_realiz}. Of special significance is the the topos $\RT{\Baire, \effBaire}$ which embodies \emph{Kleene-Vesley realizability}~\cite{KleeneSC:fouim} as well as \emph{Type 2 computability}~\cite{weihrauch00:_comput_analy}.
Luckily we need not review the construction of relative realizability toposes because the simpler categories of assemblies suffice for our purposes, as they often do.

Suppose $\AA$ is a pca and $\effAA \subseteq \AA$ an elementary sub-pca.
An \defemph{assembly} $S = (|S|, {\rz_S})$ is a set $|S|$ with a \defemph{realizability relation} ${\rz_S} \subseteq \AA \times |S|$, such that $\all{x \in |S|} \some{r \in \AA} r \rz_S x$. When $r \rz_S x$ holds we say that $r$ is a \defemph{realizer} of~$x$.
The realizability relation $\rz_S$ may be transposed in two ways to give equivalent formulations of assemblies.
First, the assembly $S$ may be seen a set $|S|$ equipped with an \defemph{existence predicate} $\mathsf{E}_S : |S| \to \pAA$, defined by $\mathsf{E}_S(x) = \set{r \in \AA \such r \rz_S x}$. Second, the assembly may be seen as a \defemph{multi-valued representation}, which is a partial map $\delta_S : \AA \parto \pow{|S|}$ defined by $\delta_S(r) = \set{x \in |S| \such r \rz_S x}$ on the subset $\set{r \in \AA \mid \some{x \in |S|} r \rz_S x}$. Thus $\mathsf{E}_S(x)$ is always inhabited, and so is $\delta_S(r)$ whenever it is defined.

An assembly $S$ is \defemph{modest} when no two elements share a realizer, i.e., if $r \rz_S x$ and $r \rz_S y$ then $x = y$. Equivalently, $S$ is modest when the representation $\delta_{S}$ is single-valued. An assembly is \defemph{partitioned} when every element has a single realizer, i.e., if $r \rz_S x$ and $s \rz_S x$ then $r = s$. Equivalently, $S$ is partitioned when $\mathsf{E}_S(x)$ is a singleton for all $x \in |S|$.

An \defemph{assembly map} $f : S \to T$ is a map between underlying sets $f : |S| \to |T|$ for which there exists $r \in \effAA$ such that if $s \rz_S x$ then $\defined{r \app s}$ and $r \app s \rz_T f(x)$. We say that $r$ \defemph{tracks} or \defemph{realizes}~$f$.

Identity maps are tracked, and so are compositions of tracked maps. Therefore
assemblies and assembly maps form a category $\Asm{\AA, \effAA}$, which turns out to be a full subcategory of the realizability topos~$\RT{\AA, \effAA}$.
It is a large category because there is an embedding $\nabla : \Set \to \Asm{\AA, \effAA}$ which maps a set~$X$ to the assembly $\nabla X = (X, {\rz_{\nabla X}})$ whose underlying set is~$|\nabla X| = X$ and the realizability relation is trivial: $r \rz_{\nabla X} x$ holds for all $r \in \AA$.

\begin{exa}%
  \label{ex:asm-N}%
  The assembly $N = (\NN, {\rz_N})$ of natural numbers has numerals as realizers,
  i.e., each $n \in \NN$ is realized by the corresponding numeral~$\numeral{n}$.
  This can be contrasted with the assembly $\nabla \NN$ in which every realizer realizes every number. Consequently, while the identity $\id : \NN \to \NN$ is realized as an assembly map $N \to \nabla \NN$, for instance by $\combK$, every assembly map $f : \nabla \NN \to N$ in the opposite direction is constant.
\end{exa}

\begin{exa}%
  \label{ex:asm-ccc}
  The category of assemblies is cartesian closed.
  The product of $S$ and $T$ is the assembly $S \times T$ given by
  \begin{gather*}
    |S \times T| = |S| \times |T|
    \\
    r \rz_{S \times T} (x, y) \iff
    \text{$\fst \app r \rz_S x$ and $\snd \app r \rz_T y$}.
  \end{gather*}
  Projections and pairing are realized respectively by $\fst$, $\snd$, and $\mathtt{pair}$.
  The exponential assembly~$T^S$ by
  \begin{gather*}
    |T^S| = \{ f : |S| \to |T| \such
                   \some{r \in \AA} \all{x \in |S|} \all{s \in \AA}
                   s \rz_S x \lthen r \app s \rz_T f(x) \}
    \\
    r \rz_{T^S} f \iff
    \all{x \in |S|} \all{s \in \AA} s \rz_S x \lthen r \app s \rz_T f(x),
  \end{gather*}
  where it should be noted that it contains maps realized by the elements of~$\AA$, not~$\effAA$. Evaluation and currying are realized respectively by $[a] (\fst \app a) \app (\snd \app a)$ and $[a] [b] [c] a \app \pair{b, c}$.
\end{exa}

\subsection{The realizability logic}%
\label{sec:realizability-logic}

In order to unravel the realizability interpretation of instance reducibilities, we first need an explicit description of realizability logic. Throughout we work with a pca~$\AA$ and an elementary sub-pca $\effAA \subseteq \AA$.

The realizability logic is embodied by \defemph{realizability predicates}. A realizability predicate on an assembly~$S$ is a map $\phi : |S| \to \pAA$. The set $\Pred{S} = \pAA^{|S|}$ of all such predicates forms a Heyting prealgebra with the preorder relation~$\leq_S$ defined by
\begin{equation*}
  \phi \leq_S \psi
  \iff
  \begin{aligned}[t]
    &\text{there is $r \in \effAA$ such that for all $x \in |S|$ and $s \in \AA$,}\\
    &\text{if $s \rz_S x$ and $t \in \phi(x)$ then $r \app s \app t \in \psi(x)$}.
  \end{aligned}
\end{equation*}
If we think of the elements of $\phi(x)$ as computational witnesses, then $\phi \leq_S \psi$ holds when witnesses of~$\phi(x)$, together with witnesses of~$x$, may be mapped to witnesses of~$\psi(x)$ using a realizer from~$\effAA$. It is customary to write $r \rz \phi(x)$ instead of $r \in \phi(x)$, and read it as ``$r$ realizes~$\phi(x)$''.

The Heyting prealgebra structure on $\Pred{S}$ is as follows, where $\phi, \psi \in \Pred{S}$, $x \in |S|$, and $r \in \AA$:
\begin{enumerate}[align=left]
\item $r \rz \bot(x)$ never,
\item $r \rz \top(x)$ always,
\item $r \rz (\phi \land \psi)(x)$ when $\fst \app r \rz \phi(x)$ and $\snd \app r \rz \psi(x)$,
\item $r \rz (\phi \lor \psi)(x)$ when either $\fst \app r = \numeral{0}$ and $\snd \app r \rz \phi(x)$, or $\fst \app r = \numeral{1}$ and $\snd \app r \rz \psi(x)$,
\item $r \rz (\phi \lthen \psi)(x)$ when for all $p \in \AA$, if $p \rz \phi(x)$ then $r \app p \rz \psi(x)$.
\end{enumerate}
These clauses say, for example, that $\bot(x) = \emptyset$, $\top(x) = \AA$, and that implication is given by
\begin{equation*}
   (\phi \lthen \psi)(x) =
   \set{r \in \AA \such \all{p \in \AA} p \in \phi(x) \lthen \defined{r \app p} \land r \app p \in \psi(x)}.
\end{equation*}
Taking into account that $\lnot \phi$ is an abbreviation for $\phi \lthen \bot$, we may calculate that $\lnot \phi$ and $\lnot \lnot \phi$ are characterized by
\begin{enumerate}[resume,align=left]
\item $r \rz (\lnot \phi)(x)$ when there is no $s \in \AA$ such that $s \rz \phi(x)$, and
\item $r \rz (\lnot\lnot \phi)(x)$ when there is $s \in \AA$ such that $s \rz \phi(x)$.
\end{enumerate}
(The connoisseurs can note that the representation of predicates given here is equivalent to the strict extensional predicates~\cite[\S2.2]{oosten08:_realiz} because extensionality is automatic for assemblies and we ensure strictness by passing to~$r$ a realizer $s \rz_S x$.)

Given a two-place realizability predicate $\rho \in \Pred{S \times T}$, we define realizability predicates~$\forall_S \rho \in \Pred{T}$ and $\exists_S \rho \in \Pred{T}$:
\begin{enumerate}[resume,align=left]
\item $r \rz (\forall_S \rho)(y)$ when for all $s, t, x$, if $s \rz_S x$ and $t \rz_T y$ then
      $r \app s \app t \rz \rho(x, y)$,
\item $r \rz (\exists_S \rho)(y)$ when for all $t$, if $t \rz_T y$ then there is $x \in |S|$ such that $\fst \app (r \app t) \rz_S x$ and $\snd \app (r \app t) \rz \rho(x, y)$.
\end{enumerate}
We prefer to write $\all{x \of S} \rho(x,y)$ and $\some{x \of S} \rho(x,y)$ instead of $(\forall_S \rho)(y)$ and $(\exists_S \rho)(y)$.
Quantification of an $(n+1)$-place predicate $\phi \in \Pred{S \times T_1 \times \cdots T_n}$ is carried out as a quantification of an equivalent two-place predicate on $S \times T$ where $T = T_1 \times \cdots \times T_n$. As a special case $n = 0$ we get quantification of a unary predicate $\phi \in \Pred{S}$:
\begin{enumerate}[resume,align=left]
\item\label{it:forall-unary} $r \rz \all{x \of S} \phi(x)$ when for all $s, x$, if $s \rz_S x$ then $r \app s \rz \phi(x)$,
\item $r \rz \some{x \of S} \phi(x)$ when there is $x \in |S|$ such that $\fst \app r \rz_S x$ and $\snd \app r \rz \phi(x)$.
\end{enumerate}
Finally, equality on $S$ is the realizability relation ${=_S} : |S| \times |S| \to \pAA$ characterized by
\begin{enumerate}[resume,align=left]
\item $r \rz x =_S y$ when $x = y$.
\end{enumerate}

\noindent
The Heyting prealgebra structure, together with quantifiers and equality given above, comprises the realizability interpretation of first-order logic with equality, which is intuitionistically sound in the sense that an intuitionistically provable statement has a realizer in~$\effAA$.

\subsection{Instance degrees in assemblies}%
\label{sec:inst-degr-assembl}

Every instance degree in a realizability topos already appears as an instance degree on an assembly, thanks to \autoref{prop:pa-equivalent-to-projective} and the following observation.

\begin{prop}%
  \label{prop:partitioned-asm-cover}
  The partitioned assemblies are (internally) projective in $\RT{\AA, \effAA}$, and every object of the topos is covered by a partitioned assembly.
\end{prop}

\begin{proof}
  The proof of~\cite[Prop.~3.2.7]{oosten08:_realiz} for the effective topos adapts to any relative realizability topos.
\end{proof}

We thus lose nothing by narrowing attention from the realizability topos to the subcategory of assemblies.
Given realizability predicates $\rpredicate{\phi}{S}$ and $\rpredicate{\psi}{T}$ on assemblies~$S$ and~$T$, let us unfold the realizability interpretation of $\phi \ileq \psi$. A realizer $r \in \effAA$ for
\begin{equation*}
  \all{x \of S} \some{y \of T} \psi(y) \lthen \phi(x)
\end{equation*}
operates as follows: for all $s, x$, if $s \rz_S x$ then there is $y \in |T|$ such that $\fst \app (r \app s) \rz_T y$ and whenever $p \rz \psi(y)$ then $\snd \app (r \app s) \app p \rz \phi(x)$.
Such $r$ can be equivalently represented by $\ell_1, \ell_2 \in \effAA$ satisfying: for all $s, x$, if $s \rz_S x$ then $\ell_1 \app s \rz_T y$ for some $y \in |T|$ and whenever $p \rz \psi(y)$ then $\ell_2 \app s \app p \rz \phi(x)$.
Indeed, from $r$ we obtain $\ell_1 = [a] \fst \app (r \app a)$ and $\ell_2 = [a] \snd \app (r \app a)$, whereas from $\ell_1$ and $\ell_2$ we reconstruct $r = [a] \pair{\ell_1 \app a, \ell_2 \app a}$.
We have established the following description of instance reducibility in assemblies.

\begin{prop}%
  \label{prop:ileq-realizability}%
  A realizability predicate $\rpredicate{\phi}{S}$ is instance reducible to $\rpredicate{\psi}{T}$ when
  there exist $\ell_1, \ell_2 \in \effAA$ such that, for all $s \in \AA$ and $x \in |S|$, if $s \rz_S x$ then there is $y \in |T|$ such that $\ell_1 \app s \rz_T y$, and for all $p$, if  $p \rz \psi(y)$ then $\ell_2 \app s \app p \rz \phi(x)$.
\end{prop}

The preceding proposition may be used to translate the constructive statements of \autoref{sec:inst-reduc} to any realizability model.
To familiarize ourselves with instance degrees in assemblies, we illustrate the method on several examples.

Because the subobject classifier $\Omega$ is not an assembly, the degree of Excluded middle is not directly represented by a realizability predicate on an assembly. However, we may lift Excluded middle along the cover $\nabla (\pAA) \to \Omega$ and calculate it to be the realizability predicate $\rpredicate{\LEM}{\nabla(\pAA)}$ given by
\begin{equation*}
  \LEM(\theta) = \set{ \pair{\numeral{n}, r} \mid (n = 1 \land r \in \theta) \lor (n = 0 \land \theta = \emptyset) }.
\end{equation*}
Thus $\LEM(\emptyset) = \set{ \pair{\numeral{0}, r} \mid r \in \AA }$, while for $\theta \neq \emptyset$ we have $\LEM(\theta) = \set{ \pair{\numeral{1}, r} \mid r \in \theta }$.
Next we characterize the $\lnot\lnot$-dense realizability predicates.

\begin{prop}%
  \label{prop:rz-not-not-dense}
  A realizability predicate $\rpredicate{\phi}{S}$ is $\lnot\lnot$-dense if, and only if, $\phi(x) \neq \emptyset$ for all $x \in |S|$.
\end{prop}

\begin{proof}
  Suppose $r \rz \all{x \of S} \lnot\lnot \phi(x)$ and consider any $x \in |S|$. There is $s \in \AA$ such that $s \rz_S x$, hence $r \app s \rz \lnot\lnot \phi(x)$ and so $\phi(x) \neq \emptyset$.
  Conversely, if $\phi(x) \neq \emptyset$ for all $x \in |S|$, then $\lnot\lnot\phi(x) = \AA$ for all $x \in |S|$,
  therefore $[a] a \rz \all{x \of S} \lnot\lnot \phi(x)$.
\end{proof}

The characterization of $\lnot\lnot$-dense degrees from \autoref{prop:dense-iff-below-lem} transfers as follows.

\begin{prop}
  A realizability predicate $\rpredicate{\phi}{S}$ is instance reducible to~$\LEM$, if and only if, $\phi(x) \neq \emptyset$ for all $x \in |S|$.
\end{prop}

\begin{proof}
  Suppose $\ell_1, \ell_2 \in \effAA$ realize the reduction $\phi \ileq \LEM$, and consider any $x \in |S|$. There is
  $r \in \AA$ such that $r \rz_S x$, hence there is $\theta \in \pAA$ such that $\ell_2 \app r  \app p \in \phi(x)$ for all $p \in \LEM(\theta)$. Now $\phi(x) \neq \emptyset$ because $\LEM(\theta) \neq \emptyset$.

  Conversely, suppose $\phi(x) \neq \emptyset$ for all $x \in |S|$. We claim that $\ell_1 = \combK \app \combK$ and $\ell_2 = \combK \app \snd$ realize $\phi \ileq \LEM$. For suppose $x \in |S|$ and $r \rz_S x$. Then $\LEM(\phi(x)) = \set{ \pair{\numeral{1}, p} \mid p \in \phi(x)}$ because $\phi(x) \neq \emptyset$ and consequently $\ell_2 \app s \app \pair{\numeral{1}, p} = p \in \phi(x)$, as required.
\end{proof}

Let us unravel the Heyting implication $\phi \ithen \psi$ of realizability predicates $\rpredicate{\phi}{S}$ and $\rpredicate{\psi}{T}$. The underlying assembly $S \ithen T$ is defined by
\begin{gather*}
  |S \ithen T| =
  \set{ (r, \theta) \in \AA \times \pow{A} \such
        r \rz \all{x \of S} \some{y \of T} \psi(y) \lthen \phi(x) \lor \theta
  },\\
  s \rz_{S \ithen T} (r, \theta) \iff s = r
\end{gather*}
and the predicate $(\phi \ithen \psi) : (S \ithen T) \to \pAA$ by
$
  (\phi \ithen \psi) (r, \theta) = \theta
$.
This is hardly an illuminating description. Hopefully, in the future there will be a better one, as well as some interesting applications of Heyting implication of instance degrees.

Suprema and infima of families of realizability predicates must be computed with proper attention to realizers. A family of assemblies ${(S_i)}_{i \in I}$ indexed by an assembly~$I$ is just an assignment of an assembly $S_i$ to each index $i \in |I|$. Similarly, a family of realizability predicates $\rpredicate{\phi_i}{S_i}$ indexed by $I$ assigns a map $\phi_i : |S_i| \to \pAA$ to each $i \in |I|$. The supremum of such a family is the realizability predicate $\ibigsup{} \phi : \coprod S \to \pAA$, where
\begin{gather*}
  \textstyle
  \coprod S = \set{ (i, x) \such i \in |I| \land x \in |S_i|},
  \\
  r \rz_{\coprod S} (i, x)
  \iff \text{$\fst \app r \rz_I i$ and $\snd \app r \rz_S x$},
\end{gather*}
is the coproduct of the family, and $\ibigsup{} \phi$ is defined by $(\ibigsup{} \phi) (i, x) = \phi_i(x)$.
The point here is that by using the coproduct $\coprod S$, an instance reduction from the supremum receives not only a realizer $x \in |S_i|$ but also one for $i \in |I|$.

\subsection{Extended Weihrauch degrees}%
\label{sec:extend-weihr-degr}

Because the category of assemblies is large, the instance degrees form a proper class. Here is an equivalent description of instance degrees which forms a set-sized preorder.

\begin{defi}
   An \defemph{extended Weihrauch predicate} is a map $f : \AA \to \ppAA$.
   Its \defemph{support} is the set $\support{f} = \set{r \in \AA \such f(r) \neq \emptyset}$.
   We say that~$f : \AA \to \ppAA$ is \defemph{Weihrauch reducible} to $g : \AA \to \ppAA$, written $f \Wleq g$, when there exist $\ell_1, \ell_2 \in \effAA$ such that for all $r \in \support{f}$:
   \begin{enumerate}
   \item $\defined{\ell_1 \app r}$ and $\ell_1 \app r \in \support{g}$,
   \item\label{it:ext-weih-l2} for every $\theta \in f(r)$ there is $\xi \in g(\ell_1 \app r)$ such that $\ell_2 \app r \rz \xi \lthen \theta$.
   \end{enumerate}
\end{defi}

\noindent
(Recall from~\autoref{sec:inst-degr-rz} that $\ell_2 \app r \rz \xi \lthen \theta$ means: if $s \in \xi$ then $\defined{\ell_2 \app r \app s}$ and $\ell_2 \app r \app s \in \theta$.)

Condition~\eqref{it:ext-weih-l2} above may look unusual to readers who are already familiar with Weihrauch reducibilities, for it looks as if the dependence between~$f$ and~$g$ is reversed. This is not the case, however, since $\ell_2 \cdot r$ still reduces $\xi \in g(\ell_1 \cdot r)$ to $\theta \in f(r)$. Also, the first part of the condition, stating that every~$\theta \in f(r)$ has some~$\xi \in g(\ell \cdot r)$, imposes no computability condition on how~$\xi$ is to be found --- it just has to exist --- which is precisely the novelty introduced by the extended degrees.
It may help to note that the embedding of ordinary Weihrauch degrees into the extended ones, see \autoref{prop:ordinary-embed-extended} below, embeds an ordinary degree in such a way that $f(r)$ and $g(\ell_1 \cdot r)$ are singletons, which collapses the first part of~\eqref{it:ext-weih-l2} so that we are left just with the familiar condition.

Weihrauch reductions evidently form a preorder, therefore its symmetrization $\Wequiv$ forms an equivalence relation, whose classes we call \defemph{extended Weihrauch degrees}.

\begin{prop}%
  \label{prop:ext-weih-equivalent-instance}
  Weihrauch reductions and instance reductions in~$\Asm{\AA, \effAA}$ are equivalent preorders.
\end{prop}

\begin{proof}
  In one direction the equivalence maps a realizability predicate $\rpredicate{\phi}{S}$ to the extended Weihrauch predicate $\mathbf{f}_{(\phi,S)} : \AA \to \ppAA$ given by
  \begin{equation*}
    \mathbf{f}_{(\phi,S)} (r) =
    \set{ \theta \in \pAA \mid \text{there is $x \in |S|$ such that $r \rz_S x$ and $\theta = \phi(x)$} }.
  \end{equation*}
  Let us verify that~$\mathbf{f}$ is monotone.
  If $\ell_1, \ell_2 \in \effAA$ witness $(\phi,S) \ileq (\psi, T)$, as in \autoref{prop:ileq-realizability}, then they witness $\mathbf{f}_{(\phi,S)} \Wleq \mathbf{f}_{(\psi,T)}$ too.
  First, if $r \in \support{\mathbf{f}_{(\phi,S)}}$ then $r \rz_S x$ for some $x \in |S|$, hence
  there is $y \in |T|$ such that $\ell_1 \app r \rz_T y$, from which we conclude $\ell_1 \app r \in \support{\mathbf{f}_{(\psi,T)}}$.
  Second, for any $\theta \in \mathbf{f}_{(\phi,S)}(r)$ there is $x \in |S|$ such that $r \rz_S x$ and $\theta = \phi(x)$, so we may take $\xi = \psi(y)$ to satisfy $\ell_2 \app r \rz \xi \lthen \theta$.

  \newcommand{\bphi}{\pmb{\phi}}%
  In the opposite direction the equivalence takes an extended Weihrauch predicate $f : \AA \to \ppAA$ to the realizability predicate $\rpredicate{\bphi_f}{\mathbf{S}_f}$ where
  \begin{gather*}
    |\mathbf{S}_f| = \set{ (r, \theta) \in \AA \times \pAA \mid \theta \in f(r) },\\
    s \rz_{\mathbf{S}_f} (r, \theta) \iff s = r,\\
    \bphi_f (r, \theta) = \theta.
  \end{gather*}
  Again, we must establish monotonicity of $(\bphi, \mathbf{S})$. If $\ell_1, \ell_2 \in \effAA$ witness $f \Wleq g$ then they also witness $(\bphi_f, \mathbf{S}_f) \ileq (\bphi_g, \mathbf{S}_g)$. First, if $s \rz_{\mathbf{S}_f} (r, \theta)$ then $s = r$ and $\theta \in f(r)$, therefore $r \in \support{f}$ and $\ell_1 \app r \in \support{g}$. Second, there is $\xi \in g(\ell_1 \app r)$ such that $\ell_2 \app r \rz \xi \lthen \theta$, hence $(\ell_1 \app r, \xi) \in |\mathbf{S}_g|$ and if $p \in \xi$ then $\ell_2 \app r \app p \in \theta$, as required.

  It remains to be checked that $\mathbf{f}$ and $(\bphi, \mathbf{S})$ form an equivalence. Unfolding of definitions reveals that $\mathbf{f}_{(\mathbf{S}_f, \bphi_f)} = f$ for all $f : \AA \to \ppAA$.
  To see that $(\bphi_{\mathbf{f}_{(\phi,S)}}, \mathbf{S}_{\mathbf{f}_{(\phi,S)}}) \iequiv (\phi, S)$ for any $\rpredicate{\phi}{S}$, observe that
  \begin{gather*}
    |\mathbf{S}_{\mathbf{f}_{(\phi,S)}}| =
       \set{(r, \theta) \in \AA \times \pAA \mid
          \text{there is $x \in |S|$ such that $r \rz_S x$ and $\theta = \phi(x)$}
       }, \\
    s \rz_{\mathbf{S}_{\mathbf{f}_{(\phi,S)}}} (r, \theta) \iff s = r.
  \end{gather*}
  Now it is clear that both $(\bphi_{\mathbf{f}_{(\phi,S)}}, \mathbf{S}_{\mathbf{f}_{(\phi,S)}}) \ileq (\phi, S)$ and $(\phi, S) \ileq (\bphi_{\mathbf{f}_{(\phi,S)}}, \mathbf{S}_{\mathbf{f}_{(\phi,S)}})$ are realized by $\ell_1 = [a] a$ and $\ell_2 = [a] [b] b$.
\end{proof}

Various properties of realizability predicates may be transferred to extended Weihrauch degrees via the equivalence given in the proof of \autoref{prop:ext-weih-equivalent-instance}. For example, if $\phi \subseteq S$ is a realizability predicate on a modest assembly then the corresponding extended Weihrauch degree $\mathbf{f}_{(\phi, S)}$ satisfies
\begin{equation*}
  \mathbf{f}_{(\phi, S)} (r) =
  \begin{cases}
    \set{\phi(x)} & \text{if $r \rz_S x$}, \\
    \emptyset & \text{if $r$ does not realize any $x \in |S|$.}
  \end{cases}
\end{equation*}
Thus we define a \defemph{modest extended Weihrauch predicate} to be a map $f : \AA \to \ppAA$ such that $f(r)$ has at most one element for every $r \in \AA$. In a similar fashion the characterization of $\lnot\lnot$-density given in \autoref{prop:rz-not-not-dense} prompts us to define a \defemph{$\lnot\lnot$-dense extended Weihrauch predicate} to be a map $f : \AA \to \ppAA$ such that $\theta \neq \emptyset$ for all $r \in \AA$ and $\theta \in f(r)$.

\subsection{Weihrauch degrees as a sublattice of instance degrees}%
\label{sec:weihr-degr-modest}

Weihrauch degrees~\cite{brattka_gherardi_2011} are defined in the context of Type~2 computability~\cite{weihrauch00:_comput_analy}.
We first generalize them in a straightforward manner to any pca~$\AA$ with an elementary sub-pca $\effAA \subseteq \AA$.

An \defemph{(ordinary) Weihrauch predicate} is a relation $U \subseteq \AA \times \AA$, whose \defemph{support} is $\support{U} = \set{r \in \AA \such \some{s \in \AA} (r, s) \in U}$.
By writing $U[r] = \set{s \in \AA \such (r, s) \in U}$ we construe~$U$ as a multi-valued map from $\support{U}$ to~$\AA$.
A \defemph{Weihrauch reduction} $U \wleq V$ from $V \subseteq \AA \times \AA$ is given by $\ell_1, \ell_2 \in \effAA$ such that, for all $r \in \support{U}$,
\begin{enumerate}
\item $\defined{\ell_1 \app r}$ and $\ell_1 \app r \in \support{V}$, and
\item for all $s \in V[\ell_1 \app r]$ we have $\defined{\ell_2 \app r \app s}$ and $\ell_2 \app r \app s \in U[r]$.
\end{enumerate}
Once again, $\wleq$ is a preorder whose symmetrization $\wequiv$ is an equivalence relation. Its classes are the \defemph{(ordinary) Weihrauch degrees}.

\begin{prop}%
  \label{prop:ordinary-embed-extended}
  Ordinary Weihrauch reductions form a preorder that is equivalent to extended Weihrauch reductions on the $\lnot\lnot$-dense modest extended Weihrauch predicates.
\end{prop}

\begin{proof}
  To each Weihrauch predicate $U \subseteq \AA \times \AA$ we associate the extended Weihrauch predicate $\extw{U} : \AA \to \ppAA$, defined by
  \begin{equation*}
    \extw{U}(r) =
    \begin{cases}
      \set{U[r]} & \text{if $r \in \support{U}$,}\\
      \emptyset  & \text{otherwise}.
    \end{cases}
  \end{equation*}
  It is clear that $\extw{U}$ is both $\lnot\lnot$-dense and modest.
  The assignment is monotone, for if $U \wleq V$ is witnessed by $\ell_1, \ell_2 \in \effAA$ then $\extw{U} \Wleq \extw{V}$ is witnessed by $\ell_1$ and $\ell_2$ as well.

  For the opposite direction, suppose $\phi : \AA \to \ppAA$ is an extended Weihrauch predicate which is $\lnot\lnot$-dense and modest, so that there is a unique map $u : \support{\phi} \to \pAA \setminus \set{\emptyset}$ such that $\phi(r) = \set{u(r)}$ for all $r \in \support{\phi}$. Let $U_\phi \subseteq \AA \times \AA$ be the Weihrauch predicate characterized by $\support{U_\phi} = \support{\phi}$ and $U_\phi[r] = u(r)$, i.e.,
  \begin{equation*}
    U_\phi = \set{ (r, s) \in \AA \times \AA \such \some{\theta \in \phi(r)} s \in \theta }.
  \end{equation*}
  Once again it is easy to see that if $\ell_1, \ell_2 \in \effAA$ witness $\phi \Wleq \psi$ then they also witness $U_\phi \wleq U_\psi$.
  The maps $U \mapsto \extw{U}$ and $\phi \mapsto U_\phi$ are monotone and inverses of each other, therefore they constitute an equivalence of preorders.
\end{proof}


\section{Examples}%
\label{sec:examples}

We make a cursory exploration of the structure of instance degrees in realizability models by way of several examples, and leave a more serious analysis for a future time.

\subsection{\texorpdfstring%
{Non-trivial non-$\lnot\lnot$-degrees}%
{Non-trivial non-not-not-degrees}%
}%
\label{sec:non-trivial-non}

When $\effAA = \AA$ the only degree which is not $\lnot\lnot$-dense is the top degree. Indeed, if $\rpredicate{\phi}{S}$ is not $\lnot\lnot$-dense then $\phi(a) = \emptyset$ for some $a \in |S|$, and there is $r \in \AA$ such that $r \rz_S a$.
Crucially, $r \in \effAA$ allows us to use $\pair{r, \combK} \in \effAA$ as a realizer for the statement $\some{x \of S} \lnot \phi(x)$, which claims that~$\phi$ has a counter-example. Therefore $\phi$ represents the top degree.
The same trick does not work when $\effAA \neq \AA$, for we have no guarantee that there is $r \in \AA'$ such that $r \rz_S a$.

The previous observation inspires the following example. Recall from \autoref{ex:asm-N} and \autoref{ex:asm-ccc} that the exponential assembly $N^N = F = (|F|, {\rz_F})$ is the set of functions $\NN \to \NN$ that are realized by elements of~$\AA$:
\begin{gather*}
  |F| = \set{f : \NN \to \NN \such \some{r \in \AA} \all{n \in \NN} r \app \numeral{n} = \numeral{f(n)} },\\
  r \rz_F f \iff \all{n \in \NN} r \app \numeral{n} = \numeral{f(n)}.
\end{gather*}
Let $\varphi$ be a standard enumeration of partial computable maps, and
define the realizability predicate $\rpredicate{\CT}{F}$ by
\begin{equation*}
  \CT(f) = \set{\overline{m} \in \AA \mid f = \varphi_m}.
\end{equation*}
Here $\CT$ stands for ``Church's thesis'' because the realizers of $\CT(f)$ are numerals $\overline{m}$ witnessing Turing-computability of~$f$.
Construed as an extended Weihrauch predicate, $\CT$ is the map $\CT : \AA \to \ppAA$ defined by
\begin{equation*}
  \textsf{CT}(r) =
  \begin{cases}
    \set{\set{\overline{m} \in \AA \mid \all{k \in \NN} r \app \overline{k} = \varphi_m(k)}} &
     \text{if $r \app \overline{k}$ is a numeral for all $k \in \NN$,}
    \\
    \emptyset & \text{otherwise.}
  \end{cases}
\end{equation*}
The formula $\all{f \of F} \CT(f)$, which can be read as ``all functions are computable'', is realized by $s \in \effAA$ such that if $r \rz_F f$ then $s \app r = \overline{m}$ and $f = \varphi_m$ for some $m \in \NN$. Whether such a realizer exists depends on the choice of~$\effAA$ and $\AA$. For example, in the effective topos $\effAA = \AA = \NN$ and we may simply take $s = [a] a$. In contrast, in Kleene-Vesley realizability $\CT$ is an interesting representative of a degree that does not arise as an ordinary Weihrauch degree. Indeed it is not $\lnot\lnot$-dense because $\CT(h) = \emptyset$ for non-computable $h : \NN \to \NN$, but at the same time $\some{h \of F} \lnot \CT(h)$ has no computable realizers.

\subsection{Embedding of truth values}%
\label{sec:embedd-truth-valu}

Let us reformulate the embeddings of the subobject classifier into instance degrees from~\autoref{sec:truth-values} as extended Weihrauch degrees.
A short calculation reveals that the embedding from \autoref{prop:omega-monotone} takes a realizability truth value $\theta \in \pAA$ to the (ordinary) Weihrauch predicate $T_\theta : \AA \to \ppAA$, defined by
\begin{equation*}
  T_\theta(r) =
  \begin{cases}
    \set{\AA} & \text{if $r \in \theta$,} \\
    \emptyset & \text{otherwise.}
  \end{cases}
\end{equation*}
It is easily seen that $T_\theta \Wleq T_\xi$ holds precisely when there is $\ell \in \effAA$ such that $\ell \rz \theta \lthen \psi$, which reaffirms monotonicity of the embedding with respect to~$\lthen$.
In Kleene-Vesley realizability $\theta \mapsto T_\theta$ is just the familiar embedding of the Medvedev lattice into Weihrauch degrees.

The embedding from \autoref{prop:omega-antimonotone} takes $\theta \in \pAA$ to the modest extended Weihrauch predicate $\theta^\star : \AA \to \ppAA$, defined by
$
  \theta^\star(r) = \set{\theta}
$,
which in Kleene-Vesley realizability provides an anti-monotone semilattice embedding of the Medvedev lattice into extended Weihrauch degrees. It \emph{almost} maps into ordinary Weihrauch degrees, as only $\emptyset^\star(r) = \set{\emptyset}$ fails to be $\lnot\lnot$-dense.

\subsection{Reductions to and from non-modest degrees}%
\label{sec:non-unif-reduct}

The degrees from \autoref{sec:non-trivial-non} and \autoref{sec:embedd-truth-valu} are all modest, so one may wonder whether the
non-modest degrees are of any use. We show examples that convey how non-modest degrees allow us to control
an aspect of reductions that one might refer to as \emph{uniformity}.

A map $\phi : \NN \to \pAA$ may be seen either as a realizability predicate on the modest assembly~$N$ of natural numbers, as defined above, or on the non-modest assembly $\nabla \NN$.
Now, assuming we have another such map $\psi : \NN \to \pAA$, there are four possibilities:
\begin{enumerate}

\item\label{it:N-N}
  $(\phi, N) \ileq (\psi, N)$ when there are $\ell_1, \ell_2 \in \effAA$ such that, for every $m \in \NN$, there is (a unique) $n \in \NN$ with  $\ell_1 \app \numeral{m} = \numeral{n}$ and $\ell_2 \app \numeral{m} \rz \psi(n) \lthen \phi(m)$. This case is analogous to the ordinary Weihrauch reductions.

\item\label{it:nabla-N}
  $(\phi, \nabla \NN) \ileq (\psi, N)$ when there are $\ell_2 \in \effAA$ and $n \in \NN$ such that for all $m \in \NN$ we have $\ell_2 \rz \psi(n) \lthen \phi(m)$.
  This case is somewhat pathological, as it happens precisely when $\psi \iequiv \itop$ or $\phi \iequiv \true{\one}$.

\item\label{it:N-nabla}
  $(\phi, N) \ileq (\psi, \nabla \NN)$ when there is a map $f : \NN \to \NN$ and $\ell_2 \in \effAA$ such that for every $m \in \NN$ we have $\ell_2 \app \overline{m} \rz \psi(f(m)) \lthen \phi(m)$.
  This case is \emph{non-uniform} in the sense that~$\ell_1$ is replaced by a map~$f$ that need not be realized. An additional example of this kind is given below.

\item\label{it:nabla-nabla}
  $(\phi, \nabla \NN) \ileq (\psi, \nabla \NN)$ when there is a map $f : \NN \to \NN$ and $\ell_2 \in \effAA$ such that for every $m \in \NN$ $\ell_2 \rz \psi(f(m)) \lthen \phi(m)$. In contrast to the previous case, $\ell_2$ must work for every~$m$ \emph{without} being given it as an input.
\end{enumerate}
To summarize, a reduction $\phi \ileq \psi$ from a modest~$\phi$ to a non-modest~$\psi$ may be non-uniform in the sense that it need not compute an instance of~$\psi$ to which a given instance of~$\phi$ reduces.
In the opposite direction, a reduction $\phi \ileq \psi$ from a non-modest~$\phi$ to a modest~$\psi$ must be uniform in the sense that all instances of $\phi$ are reduced to the same instance of~$\psi$.

The above examples illustrate just the most extreme possibilities. We may calibrate the uniformity aspects of
reductions by using assemblies that are neither modest nor in the image of~$\nabla$.

\medskip

Let us conclude with another example of reductions between modest and non-modest degrees. This time we work directly with (extended) Weihrauch degrees in Kleene-Vesley realizability.
The principle of Weak Excluded Middle
\begin{equation*}
  \all{p \of \Omega} \lnot p \lor \lnot\lnot p
\end{equation*}
is represented by the (non-modest) extended Weihrauch predicate $\WLEM : \Baire \to \pow{\pow{\Baire}}$,
\begin{equation*}
  \WLEM(\alpha) = \set{\set{\numeral{0}}, \set{\numeral{1}}},
\end{equation*}
while the Limited principle of omniscience~\eqref{eq:lpo} is represented by the (modest) Weihrauch predicate
$\LPO : \Baire \to \pow{\pow{\Baire}}$,
\begin{equation*}
  \LPO(\alpha) =
  \begin{cases}
    \set{\set{\numeral{0}}} & \text{if $\all{n} \alpha_n = 0$} \\
    \set{\set{\numeral{1}}} & \text{if $\some{n} \alpha_n \neq 0$}.
  \end{cases}
\end{equation*}
The reduction $\LPO \ileq \WLEM$ is witnessed by computable maps $\ell_1 : \Baire \to \Baire$ and
$\ell_2 : \Baire \times \set{\numeral{0}, \numeral{1}} \to \set{\numeral{0}, \numeral{1}}$ such that, for every $\alpha \in \Baire$:
\begin{enumerate}[(a)]
\item $\ell_1(\alpha) \in \Baire$, and
\item for every $\theta \in \LPO(\alpha)$ there is $\xi \in \set{\set{\numeral{0}}, \set{\numeral{1}}}$ such that
  $\ell_2(\alpha, b) \in \theta$ for all $b \in \xi$.
\end{enumerate}
We may drop $\ell_1$ and simplify the two conditions to a single one:
\begin{enumerate}[(a),resume]
\item\label{it:lpo-ileq-wlem}
  if $\all{n} \alpha_n = 0$ then $\ell_2(\alpha, \numeral{0}) = \numeral{0}$ and if $\some{n} \alpha_n = 1$ then $\ell_2(\alpha, \numeral{1}) = \numeral{1}$.
\end{enumerate}
The map $\ell_2(\alpha, b) = b$ clearly satisfies the condition, hence $\LPO \Wleq \WLEM$ holds. Notice how the non-modesty allowed us to avoid computing which of the two possibilities in~\ref{it:lpo-ileq-wlem} happens.

How about the opposite reduction $\WLEM \Wleq \LPO$? It would be witnessed by computable $\ell_1$ and $\ell_2$ such that, for all $\alpha \in \Baire$:
\begin{enumerate}[(i)]
\item $\ell_1(\alpha) \in \Baire$, and
\item for every $\theta \in \set{\set{\numeral{0}}, \set{\numeral{1}}}$ there is $\xi \in \LPO(\ell_1(\alpha))$ such that $\ell_2(\alpha, b) \in \theta$ for all $b \in \xi$.
\end{enumerate}
The second condition is contradictory, as it requires both $\ell_2(\alpha, b) = \numeral{0}$ and $\ell_2(\alpha, b) = \numeral{1}$. Thus there is no such reduction.


\subsubsection*{Acknowledgment}
An initial version of the material presented here was carried out in cooperation with Kazuto Yoshimura from the Japan Advanced Institute of Science and Technology. Unfortunately, I have not had the opportunity to complete the work with Mr.~Yoshimura, as he left the institute and I lost all contact with him.

\bibliographystyle{alphaurl}
\bibliography{references.bib}

\end{document}